\documentclass{aptpub}

\authornames{D. FEKETE \emph{ET AL.}} 
\shorttitle{Coupled SDEs for skeletally decomposed superprocesses} 

\newcommand{\e}{\mathrm{e}}
  \def\d{{\textnormal d}}

\usepackage{fullpage}
\usepackage{graphics}
\usepackage{url}

\numberwithin{equation}{section}  

\begin{document}

\title{Skeletal stochastic differential equations\\ for superprocesses} 

\authorone[University of Exeter]{Dorottya Fekete} 
\authortwo[Universidad de Chile]{Joaquin Fontbona}
\authorthree[University of Bath]{Andreas E. Kyprianou}

\addressone{College of Engineering, Mathematics and Physical Sciences, Harrison Building, Streatham Campus, University of Exeter, North Park Road, Exeter EX4 4QF, UK. Email address: d.fekete@exeter.ac.uk} 
\addresstwo{Centre for Mathematical Modelling, DIM CMM, UMI 2807 UChile-CNRS, Universidad de Chile,  Santiago, Chile. Email address: fontbona@dim.uchile.cl} 
\addressthree{Department of Mathematical Sciences, University of Bath, Claverton Down, Bath, BA2 7AY, UK. Email address: a.kyprianou@bath.ac.uk}

\vspace{-2.5cm}
\begin{abstract}
It is well understood that a supercritical superprocess is equal in law to a discrete Markov branching process whose genealogy is dressed in a Poissonian way with immigration which initiates subcritial superprocesses. 
The Markov branching process corresponds to the genealogical description of {\it prolific individuals}, that is individuals who produce eternal genealogical lines of decent, and is often referred to as the {\it skeleton} or {\it backbone} of the original superprocess. 
The Poissonian dressing along the skeleton may be considered to be the remaining non-prolific genealogical mass in the superprocess.
Such skeletal decompositions are equally well understood for continuous-state branching processes (CSBP).

In a previous article, \cite{FFK17}, we developed an SDE approach to study the skeletal representation of CSBPs, which provided a common framework for the skeletal decompositions of supercritical and (sub)critical CSBPs.
It also helped us to understand how the skeleton thins down onto one infinite line of descent when conditioning on survival until larger and larger times, and eventually forever.  

Here our main motivation is to show the robustness of the SDE approach by expanding it to the spatial setting of superprocesses.
The current article only considers supercritical superprocesses, leaving the subcritical case open. 
\end{abstract}



\section{Introduction}

In this paper we revisit the notion of the so-called skeletal decomposition of superprocesses.
It is well-known that when the survival probability is not 0 or 1, then non-trivial infinite genealogical lines of descent, which we call \textit{prolific}, can be identified on the event of survival.
By now it is also well understood that the process itself can be decomposed along its prolific genealogies, where non-prolific mass is immigrated in a Poissonian way along the stochastically `thinner' prolific skeleton.
This fundamental phenomenon was first studied by Evans and O'Connell \cite{EC94} for superprocesses with quadratic branching mechanism.
They showed that the distribution of the superprocess at time $t\geq 0$ can be written as the sum of two independent processes.
The first is a copy of the original process conditioned on extinction, while the second process is understood as the superposition of mass that has immigrated continuously along the trajectories of a dyadic branching particle diffusion, which is initiated from a Poisson number of particles.
This distributional decomposition was later extended to the spatially dependent case by Engl\"{a}nder and Pinsky \cite{EP99}.

A pathwise decomposition for superprocesses with general branching mechanism was provided by Berestycki et al. \cite{BKM11}.
Here the role of the skeleton is played by a branching particle diffusion that has the same motion generator as the superprocess, and the immigration is governed by three independent Poisson point processes.
The first one results in what we call continuous immigration along the skeleton, where the so-called excursion measure plays the central role, and it assigns zero initial mass to the immigration process.
The second point process discontinuously grafts independent copies of the original process conditioned on extinction on to the path of the skeleton.
Finally, additional copies of the original process conditioned on extinction are immigrated off the skeleton at its branch points, where the initial mass of the immigrant depends on the number of offspring at the branch point.
The spatially dependent version of this decomposition was considered in \cite{KPR14} and \cite{EKW15}.

Other examples of skeletal decompositions for superprocesses include \cite{SV99, EW03, KR12, MSP15, FPPP18}.
\medskip

In a previous article \cite{FFK17} we developed a stochastic differential equation (SDE) approach to study the skeletal decomposition of continuous state branching processes (CSBPs).
These decompositions were by no means new; prolific genealogies for both supercritical and subcritical CSBPs had been described, albeit in the latter case we have to be careful what we mean by `prolific'.
In particular, in \cite{BFM08}, \cite{DW07} and \cite{KR12} specifically CSBPs were considered, but since the total mass process of a superprocess with spatially independent branching mechanism is a CSBP, skeletal decompositions for CSBPs also appear as a special case of some of the previously mentioned results.

The results in \cite{FFK17} were motivated by the work of Duquesne and Winkel \cite{DW07}, and Duquesne and Le Gall \cite{DG05}.
Duquesne and Winkel,  in the context of L\'evy trees, provided a parametric family of decompositions for finite-mean supercritical CSBPs that satisfy Grey's condition.
They showed that one can find a decomposition of the CSBP for a whole family of embedded skeletons, where  the 'thinnest' one is the prolific skeleton with all the infinite genealogical lines of descent, while the other embedded skeletons not only contain the infinite genealogies, but also some finite ones grafted on to the prolific tree.
On the other hand, Duquesne and Le Gall studied subcritical CSBPs, and using the height process gave a description of those genealogies who survive until some fixed time $T>0$.
It is well known that a subcritical CSBP goes extinct almost surely, thus prolific individuals, in the classic sense, do not exist in the population. 
But since it is possible that the process survives until some fixed time $T$, individuals who have at least one descendent at time $T$ can be found with positive probability. 
We call these individuals $T$-prolific.

The SDE approach provides a common framework for the parametric family of decompositions of Duquesne and Winkel, as well as for the time-inhomogeneous decompositions we get, when we decompose the process along its $T$-prolific genealogies.
We note that these finite-horizon decompositions exist for both supercritical and subcritical process.
In the subcritical case the SDE representation can be used to observe the behaviour of the system when we condition on survival up to time $T$, then take $T$ to infinity.
Conditioning a subcritical CSBP to survive eternally results in what is known as a spine decomposition, where independent copies of the original process are grafted on to one infinite line of descent, that we call the spine (for more details, we refer the reader to \cite{RR89,L07,L08,FF12,AD12}).
And indeed, in \cite{FFK17} we see how the skeletal representation becomes, in the sense of weak convergence, a spinal decomposition when conditioning on survival, and in particular how the skeleton thins down to become the spine as $T\rightarrow\infty$.

\medskip

In this paper our objective is to demonstrate the robustness of this aforementioned method by expanding the SDE approach to the spatial setting of superprocesses.
We consider supercritical superprocesses with space dependent branching mechanism, but in future work we hope to extend results to the time-inhomogeneous case of subcritical processes.

The rest of this paper is organised as follows.
In the remainder of this section we introduce our model and fix some notation. 
Then in Section \ref{2} we remind the reader of some key existing results relevant to the subsequent exposition, in particular we recall the details of the skeletal decomposition of superprocesses with spatially dependent branching mechanism, as appeared in \cite{KPR14} and \cite{EKW15}.
The main result of the paper is stated in Section \ref{3}, where we reformulate the result of Section \ref{2} by writing down a coupled SDE, whose second coordinate corresponds to the skeletal process, while the first coordinate describes the evolution of the total mass in system.
In Sections \ref{4}, \ref{5} and \ref{6}, we give the proof of our results. 
\medskip

\noindent{\bf Superprocess.}
Let $E$ be a domain of $\mathbb{R}^d$ and denote by $\mathcal{M}( E)$ the space of finite Borel measures on $ E$.
Furthermore let $\mathcal{M}( E)^\circ:=\mathcal{M}( E)\setminus\{ 0\}$, where $0$ is the null measure.
We are interested in a strong Markov process $X$ on $ E$ taking values in $\mathcal{M}( E)$.
The process is characterised by two quantities $\mathcal{P}$ and $\psi$.
Here $(\mathcal{P}_t)_{t\geq 0}$ is the semigroup of an $\mathbb{R}^d$-valued diffusion killed on exiting $ E$, and $\psi$ is the so-called branching mechanism.
The latter takes the form
\begin{equation}\label{local}
\psi(x,z)=-\alpha(x)z+\beta(x)z^2+\int_{(0,\infty)} \left( \e^{-zu}-1+zu\right)m(x,\d u),\quad x\in E,\; z\geq 0,
\end{equation}
where $\alpha$ and $\beta\geq 0$ are bounded measurable mappings from $ E$ to $\mathbb{R}$ and $[0,\infty)$ respectively, and $(u\wedge u^2)m(x,\d u)$ is a bounded kernel from $ E$ to $(0,\infty)$.

For technical reasons we assume that $\mathcal{P}$ is a Feller semigroup whose generator takes the form 
\begin{equation}\label{L}
\mathcal{L}=\frac{1}{2}\sum_{i,j=1}^d a_{ij}(x)\frac{\partial^2}{\partial x_i\partial x_j}+\sum_{i=1}^d b_i(x)\frac{\partial}{\partial x_i},
\end{equation}
where $a:E\rightarrow \mathbb{R}^{d\times d}$ is the diffusion matrix that takes values in the set of symmetric, positive definite matrices, and $b:E\rightarrow \mathbb{R}^d$ is the drift term.%

Then the one-dimensional distributions of $X$ can be characterised as follows.
For all $\mu\in \mathcal{M}( E)$ and $f\in B^+( E)$, where $B^+(E)$ denotes the non-negative measurable functions on $E$, we have
\[
\mathbb{E}_\mu\left[\e^{-\langle f, X_t\rangle}\right]=\exp\left\lbrace -\langle u_f(\cdot,t),\mu\rangle\right\rbrace,
\]
where $u_f(x,t)$ is the unique non-negative, locally bounded solution to the integral equation
\begin{equation}\label{inteq}
u_f(x,t)=\mathcal{P}_t[f](x)-\int_0^t \d s\cdot \mathcal{P}_s[\psi(\cdot,u_f(\cdot,t-s))](x),\quad x\in E,\; t\geq 0.
\end{equation}
Here we use the notation
\[
\langle f,\mu \rangle=\int_{ E}f(x)\mu(\d x),\quad \mu\in \mathcal{M}( E),\; f\in B^+(E).
\]
For each $\mu\in \mathcal{M}( E)$ we denote by $\mathbb{P}_\mu$ the law of the process $X$ issued from $X_0=\mu$.
The process $(X,\mathbb{P}_\mu)$ is called a $(\mathcal{P},\psi)$-superprocess.

For more details on the above see Fitzsimmons \cite{F88}; for a general overview on superprocesses we refer the reader to the books of Dynkin \cite{D94, D02}, Etheridge \cite{E00}, Le Gall \cite{LG99} and Li \cite{ZL11}.

\bigskip

Next, we recall the SDE representation of $(X,\mathbb{P}_\mu)$ (for more details see Chapter 7 of \cite{ZL11}).
Recall that $m$ was previously defined in \eqref{local}. We assume that 
 it satisfies the integrability condition
\[
\sup_{x\in E}\int_{(0,\infty)}(u \wedge u^2)m(x,\d u)<\infty.
\]

Let $C_0( E)^+$ denote the space of non-negative continuous functions on $ E$ vanishing at infinity.
We assume that $\alpha$ and $\beta$ are continuous,
furthermore $x\mapsto(u \wedge u^2)m(x,\d u)$ is continuous in the sense of weak convergence, and 
\[
f \mapsto\int_{(0,\infty)}( u f(x)\wedge u^2 f(x)^2)m(x,\d u)
\]
maps $C_0( E)^+$ into itself.

Next define $\Delta X_s=X_s-X_{s-}$. As a random measure difference, if $s>0$ is such that $\Delta X_s\neq 0$, it can be shown that $\Delta X_s = u_s \delta_{x_s}$ for some $u_s\in(0,\infty)$ and $x_s\in E$.
Suppose that for the countable set of times, say  $(s_i, i\in\mathbb{N})$, that $\Delta X_{s_i}\neq 0$, $i\in\mathbb{N}$, we enumerate the pairs $((u_i,x_i), i\in\mathbb{N})$. We say that 
$N(\d s, \d x, \d u), s\geq 0$ is the  optional random measure on $[0,\infty)\times E\times(0,\infty)$, which can otherwise be identified as $\sum_{i\in\mathbb{N}} \delta_{(s_i,  x_i, u_i)} (\d s, \d x, \d u)$.
 Let $\hat{N}( \d s,\d x, \d u)$ denote the predictable compensator of $N(\d s,\d x, \d u)$.
It can be shown that $\hat{N}( \d s,\d x, \d u)=\d s K(X_{s-}, \d x, \d u)$, where, given $\mu\in\mathcal{M}(E)$, 
\[
K(\mu,\d x, \d u)= \mu(\d x)m(x,\d u), \qquad x\in E, u\in (0,\infty).
\]

If we denote the compensated measure by $\tilde{N}$, then for any $f\in D_0(\mathcal{L})$ (the set of functions in $C_0(E)$ that are also in the domain of $\mathcal{L}$) we have
\begin{equation}\label{zlsde}
\langle f,X_t\rangle=\langle f,X_0\rangle+M_t^c(f)+M_t^d(f)+\int_0^t \langle \mathcal{L}f+\alpha f, X_s\rangle\d s,\quad t\geq 0,
\end{equation}
where $t\mapsto M_t^c(f)$ is a continuous local martingale with quadratic variation $2\langle \beta f^2,  X_{t-}\rangle \d t$ and
\[
t\mapsto M_t^d(f)=\int_0^t\int_{E}\int_{(0,\infty)}\langle f, u\delta_x \rangle\tilde{N}( \d s,\d x, \d u), \quad t\geq 0,
\]
is a purely discontinuous local martingale. Here and throughout the paper, we prefer to write $\langle f, u\delta_x \rangle$ in place of $uf(x)$ as a reminder that it is the increment of the process $\langle f, X_t\rangle$, $t\geq 0$.

The representation \eqref{zlsde} is what we will use in Section \ref{3} when developing the SDE approach to the skeletal decomposition of $(X,\mathbb{P}_\mu)$.
However before we  proceed with this line of analysis, we first need to recall the details of this skeletal decomposition, as it not only motivates our results, but also proves to be helpful in understanding the structure of our SDE.

\section{Skeletal decomposition}\label{2}

Recall, that the main idea behind the skeletal decomposition is that under certain conditions we can identify prolific genealogies in the population, and by immigrating non-prolific mass along the trajectories of these prolific genealogies we can recover the law of the original superprocess.
The infinite genealogies are described by a Markov branching process whose initial state is given by a Poisson random measure, while traditionally the immigrants are independent copies of the original process conditioned to become extinct.

In this section we first characterise the two components, then explain how to construct the skeletal decomposition from these building blocks. 
The results of this section are lifted from \cite{KPR14} and \cite{EKW15}.

\medskip

As we have mentioned the skeleton is often constructed using the event of extinction, that is the event $\mathcal{E}_{{\rm fin}}=\{\langle 1, X_t\rangle=0\textrm{ for some }t>0\}$. 
This guides the skeleton particles into regions where the survival probability is high.
If we write $w(x)=-\log\mathbb{P}_{\delta_x}(\mathcal{E_{{\rm fin}}})$, and assume that $\mu\in\mathcal{M}(E)$ is such that $\langle w,\mu\rangle<\infty$, then it is not hard to see that 
$
\mathbb{P}_\mu(\mathcal{E}_{{\rm fin}})=\exp\left\lbrace  -\langle w, \mu\rangle \right\rbrace.
$
Furthermore, by conditioning $\mathcal{E}_{{\rm fin}}$ on $\mathcal{F}_t:=\sigma(X_s,s\leq t)$ we get that
\[
\mathbb{E}_\mu\left(\e^{-\langle w,X_t\rangle}\right)=\e^{-\langle w,\mu \rangle}.
\]
In \cite{EKW15} the authors point out that in order to construct a skeletal decomposition along those genealogies that avoid the behaviour specified by $w$ (in this case `extinction'), all we need is that the function $w$ gives rise to a multiplicative martingale $\left(\left(\e^{-\langle w,X_t\rangle}, t\geq 0\right),\mathbb{P}_\mu\right)$.
In particular, a skeletal decomposition is given for any choice of a martingale function $w$ which satisfies the following conditions.
 \begin{itemize}
 \item For all $x\in E$ we have $w(x)>0$ and $\sup_{x\in E} w(x)<\infty $, and
 \item $\mathbb{E}_\mu\left(\e^{-\langle w,X_t\rangle}\right)=\e^{-\langle w,\mu \rangle}$ for all $\mu\in\mathcal{M}_c(E),\; t\geq 0$. (Here $\mathcal{M}_c(E)$ denotes the set of finite, compactly supported measures on $E$).
 \end{itemize}
The condition $w(x)>0$ implicitly hides the notion of supercriticality, as it ensures that survival happens with positive probability.
Note however that `survival' can be interpreted in many different ways. 
For example, the choice of $\mathcal{E}_{\rm fin}$ results in skeleton particles that are simply part of some infinite genealogical line of descent, but we could also define surviving genealogies as those who visit a compact domain in $E$ infinitely often.

\begin{remark}\rm
The authors in \cite{KPR14} and \cite{EKW15} show the existence of the skeletal decomposition under a slightly more general setup, where $w$ is only locally bounded from above.  Note, however, that their proof consists of first establishing dealing with the case when $w$ is uniformly bounded, and then appealing to a localisation argument to relax this to the aforesaid local boundedness.  
Our SDE approach requires the case of uniform boundedness, however a localisation process can in principle be used to relax the assumption as in the aforementioned literature. 
\end{remark}

We will also make the additional assumption that $w$ is in the domain of the generator $\mathcal{L}$. This is predominantly because of the use of partial differential equations in our analysis rather than integral equations. 


\subsection{Skeleton}
First we identify the branching particle system that takes the role of the skeleton in the decomposition of the superprocess.
In general, a Markov branching process $Z=(Z_t,t\geq 0)$ takes values in $\mathcal{M}_a(E)$ (the set of finite, atomic measures in $E$), and it can be characterised by the pair $(\mathcal{P},F)$, where $\mathcal{P}$ is the semigroup of a diffusion and $F$ is the branching generator which takes the form
\[
F(x,s)=q(x)\sum_{n\geq 0}p_n(x)(s^n-s),\quad x\in E,\; s\in[0,1].
\]
Here $q$ is a bounded, measurable mapping from $E$ to $[0,\infty)$, and $\{ p_n(x),n\geq 0\}$, $x\in E$ are measurable sequences of probability distributions.
For $\nu\in\mathcal{M}_a(E)$ we denote the law of the process $Z$ issued from $\nu$ by $\mathbf{P}_\nu$.
Then we can describe $(Z,\mathbf{P}_\nu)$ as follows.
We start with initial state $Z_0=\nu$. 
Particles move according to $\mathcal{P}$, and at a spatially dependent rate $q(x)\d t$ a particle is killed and is replaced by $n$ offspring with probability $p_n(x)$.
The offspring particles then behave independently and according to the same law as their parent.

In order to specify the parameters of $Z$ we first need to introduce some notation.
Let $\xi=(\xi_t,t\geq 0)$ be the diffusion process on $E\cup \{\dagger\}$ (the one-point compactification of $E$ with a cemetery state) corresponding to $\mathcal{P}$, and let us denote its probabilities by $\{\Pi_x,x\in E\}$.
(Note that the previously defined martingale function $w$ can be extended to $E\cup \{\dagger\}$ by defining $w(\dagger)=0$).
Then for all $x\in E$
\[
\frac{w(\xi_t)}{w(x)}\exp\left\lbrace -\int_0^t\frac{\psi(\xi_s,w(\xi_s))}{w(\xi_s)}\d s\right\rbrace,\quad  t\geq 0,
\]
is a positive local martingale, and hence a supermartingale. 
(To see why this is true we refer the reader to the discussion in Section 2.1.1. of \cite{EKW15}).
Now let $\tau_E=\inf\{ t>0:\xi_t\in\{\dagger\}\}$, and consider the following change of measure
\begin{equation}\label{Piw}
\left.\frac{\d\Pi_x^w}{\d \Pi_x}\right|_{\sigma(\xi_s,s\in[0,t])}=\frac{w(\xi_t)}{w(x)}\exp\left\lbrace -\int_0^t \frac{\psi(\xi_s,w(\xi_s))}{w(\xi_s)}\d s\right\rbrace,\quad \textrm{on }\{t<\tau_E\},\;x\in E,
\end{equation}
which uniquely determines a family of (sub)probability measures $\{\Pi_x^w,x\in E\}$ (see for example \cite{EK86}).

If we denote by $\mathcal{P}^w$ the semigroup of the $E\cup\{\dagger\}$ valued process whose probabilities are $\{\Pi^w_x,x\in E\}$, then it can be shown that the generator corresponding to $\mathcal{P}^w$ is given by
\[
\mathcal{L}^w:=\mathcal{L}_0^w-w^{-1}\mathcal{L}w=\mathcal{L}_0^w-w^{-1}\psi(\cdot,w),
\]
where $\mathcal{L}_0^w u=w^{-1}\mathcal{L}(wu)$ whenever $u$ is in the domain of $\mathcal{L}$.
Note that $\mathcal{L}^w$ is also called an $h$-transform of the generator $\mathcal{L}$ with $h=w$.
The theory of $h$-transforms for measure-valued diffusions was developed in \cite{EP99}.

Intuitively if 
\begin{equation}\label{wx}
w(x)=-\log\mathbb{P}_{\delta_x}(\mathcal{E})
\end{equation}
defines a martingale function with the previously introduced conditions for some tail event $\mathcal{E}$, then the motion associated to $\mathcal{L}^w$ forces the particles to avoid the behaviour specified by $\mathcal{E}$.
In particular when $\mathcal{E}=\mathcal{E}_{{\rm fin}}$ then $\mathcal{P}^w$ encourages $\xi$ to visit domains where the global survival rate is high. 

Now we can characterise the skeleton process of $(X,\mathbb{P}_\mu)$ associated to $w$.
In particular, $Z=(Z_t,t\geq 0)$ is a Markov branching process with diffusion semigroup $\mathcal{P}^w$ and branching generator
\[
F(x,s)=q(x)\sum_{n\geq 0}p_n(x)(s^n-s),\quad x\in E,\; s\in[0,1],
\]
where
\begin{equation}\label{rate1}
q(x)=\psi'(x,w(x))-\frac{\psi(x,w(x))}{w(x)},
\end{equation}
and $p_0(x)=p_1(x)=0$, and for $n\geq 2$
\begin{equation}\label{pk}
p_n(x)=\frac{1}{w(x)q(x)}\left\lbrace \beta(x)w^2(x)\mathbf{1}_{\{ n=2\}}+w^n(x)\int_{(0,\infty)}\frac{y^n}{n!}\e^{-w(x)y}m(x,\d y)\right\rbrace.
\end{equation}
Here we used the notation
\[
\psi'(x,w(x)):=\left.{\partial_z}\psi(x,z)\right|_{z=w(x)},\quad x\in E.
\]
We refer to the process $Z$ as the $(\mathcal{P}^w, F)$ skeleton.


\subsection{Immigration}
Next we characterise the process that we immigrate along the previously introduced branching particle system.
To this end let us define the following function
\[
\psi^*(x,z)=\psi(x,z+w(x))-\psi(x,w(x)),\quad x\in E,
\]
which can be written as
\begin{equation}\label{psistar1}
\psi^*(x,z)=-\alpha^*(x)z+\beta(x)z^2+\int_{(0,\infty)}(\e^{-zu}-1+zu)m^*(x,\d u),\quad x\in E,
\end{equation}
where
\[
\alpha^*(x)=\alpha(x)-2\beta(x)w(x)-\int_{(0,\infty)}(1-\e^{-w(x)u})u\;m(x,\d u)=-\psi'(x,w(x)),
\]
and
\[
m^*(x,\d u)=\e^{-w(x)u}m(x,\d u).
\]
Note that under our assumptions $\psi^*$ is a branching mechanism of the form \eqref{local}.
We denote the probabilities of the $(\mathcal{P}, \psi^*)$-superprocess by $(\mathbb{P}^*_\mu)_{\mu\in\mathcal{M}(E)}$.

If $\mathcal{E}$ is the event associated with $w$ (see \eqref{wx}), and $\langle w,\mu\rangle<\infty$, then we have
\[
\mathbb{P}^*_{\mu}(\cdot)=\mathbb{P}_\mu(\cdot|\mathcal{E}).
\]
In particular, when $\mathcal{E}=\mathcal{E}_{\rm fin}$, then $\mathbb{P}^*_\mu$ is the law of the superprocess conditioned to become extinct.


\subsection{Skeletal path decomposition}
Here we give the precise construction of the skeletal decomposition that we introduced in a heuristic way at the beginning of this section.
Let $\mathbb{D}([0,\infty)\times \mathcal{M}(E))$ denote the space of measure valued c\`adl\`ag function.
Suppose that $\mu\in\mathcal{M}(E)$, and let $Z$ be a $(\mathcal{P}^w,F)$-Markov branching process with initial configuration consisting of a Poisson random field of particles in $E$ with intensity $w(x)\mu(\d x)$.
Next, dress the branches of the spatial tree that describes the trajectory of $Z$ in such a way that a particle at the space-time position $(x,t)\in E\times[0,\infty)$ has a $\mathbb{D}([0,\infty)\times \mathcal{M}(E))$-valued trajectory grafted on to it, say $\omega = (\omega_t, t\geq0)$, with rate
\begin{equation}\label{cdcimm}
2\beta(x)\d\mathbb{Q}_x^*(\d \omega)+\int_{(0,\infty)} y \e^{-w(x)y}m(x,\d y)\times \d\mathbb{P}^*_{y\delta_x}(\d \omega). 
\end{equation}
Here $\mathbb{Q}^*_x$ is the excursion measure on the space $\mathbb{D}([0,\infty)\times \mathcal{M}(E))$ which satisfies
\[
\mathbb{Q}^*_x\left(1-\e^{-\langle f, X_t\rangle}\right)=u^*_f(x,t)
\]
for $x\in E,\;t\geq 0$ and $f\in B^+_b(E)$ (the space of non-negative, bounded measurable functions on $E$),  where $u^*_f(x,t)$ is the unique solution to \eqref{inteq} with the branching mechanism $\psi$ replaced by $\psi^*$. (For more details on excursion measures see \cite{DK04}).
Moreover, when a particle in $Z$ dies and gives birth to $n\geq 2$ offspring at spatial position $x\in E$, with probability $\eta_n(x,\d y)\mathbb{P}^*_{y\delta_x}(\d \omega)$ an additional  $\mathbb{D}([0,\infty)\times \mathcal{M}(E))$-valued trajectory, $\omega$, is grafted on to the space-time branching point, where
\begin{equation}\label{etak}
\eta_n(x,\d y)=\frac{1}{w(x)q(x)p_n(x)}\left\lbrace \beta(x)w^2(x)\delta_0(\d y)\mathbf{1}_{\{ n=2\}}+w^n(x)\frac{y^n}{n!}\e^{-w(x)y}m(x,\d y)\right\rbrace.
\end{equation}
Overall, we have three different types of immigration processes that contribute to the dressing of the skeleton. In particular, the first term of \eqref{cdcimm} is what we call `continuous immigration' along the skeleton, while the second term is referred to as the `discontinuous immigration', and finally \eqref{etak} corresponds to the so-called `branch-point immigration'.

Now we define $\Lambda_t$ as the total mass from the dressing present at time $t$ together with the mass present at time $t$ of an independent copy of $(X,\mathbb{P}_\mu^*)$ issued at time $0$.
We denote the law of $(\Lambda, Z)$ by $\mathbf{P}_\mu$.
Then in \cite{KPR14} the authors showed that $(\Lambda,\mathbf{P}_\mu)$ is Markovian and has the same law to $(X,\mathbb{P}_\mu)$.
Furthermore, under $\mathbf{P}_\mu$, conditionally on $\Lambda_t$, the measure $Z_t$ is a Poisson random measure with intensity $w(x)\Lambda_t(\d x)$.

\section{SDE representation of the dressed tree}\label{3}

Recall that our main motivation is to reformulate the skeletal decomposition of superprocesses using the language of SDEs. 
Thus in this section, after giving an SDE representation of the skeletal process, we derive the coupled SDE for the dressed skeleton, which simultaneously describes the evolution of the skeleton and the total mass in the system. 


\subsection{SDE of the skeleton.}
We use the arguments on page 3 of \cite{JX13} to derive the SDE for the branching particle diffusion, that will act as the skeleton. 
Let $(\xi_t,t\geq 0)$ be the diffusion process corresponding to the Feller semigroup $\mathcal{P}$.
Since the generator of the motion is given by \eqref{L}, the process $\xi$ satisfies
\[
\d\xi_t=b(\xi_t)\d t+\sigma(\xi_t)\d B_t,
\]
where $\sigma:\mathbb{R}^d\rightarrow\mathbb{R}^d$ is such that $\sigma(x)\sigma^{\tt T}(x)=a(x)$ (where $\tt T$ denotes matrix transpose), and $(B_t,t\geq 0)$ is a $d$-dimensional Brownian motion (see for example Chapter 1 of \cite{P08}).

It is easy to verify that if $(\tilde{\xi}_t,t\geq 0)$ is the diffusion process under $\mathcal{P}^w$, then it satisfies
\[
\d\tilde{\xi}_t=\left( b(\tilde{\xi}_t)+\frac{\nabla w(\tilde{\xi}_t)}{w(\tilde{\xi}_t)}a(\tilde{\xi}_t) \right)\d t+\sigma(\tilde{\xi}_t)\d B_t,
\]
where $\nabla w$ is the gradient of $w$.
To simplify computations, define the function $\tilde{b}$ on $E$ given by
\[
\tilde{b}(x):=b(x)+\frac{\nabla w(x)}{w(x)}a(x).
\]

For $h\in C^2_b(E)$ (the space of bounded, twice differentiable continuous functions on $E$), using It\^{o}'s formula (see e.g. Section 8.3 of \cite{BO03}) we get
\[
\d h(\tilde{\xi}_t)=(\nabla h(\tilde{\xi}_t))^{\tt T}\tilde{b}(\tilde{\xi}_t)\d t+\frac{1}{2}\mathrm{Tr}\left[ \sigma^{\tt T}(\tilde{\xi}_t) H_{h}(\tilde{\xi}_t)\sigma(\tilde{\xi}_t)\right]\d t+(\nabla h(\tilde{\xi}_t))^{\tt T} \sigma(\tilde{\xi}_t)\d B_t,
\]
where $x^{\tt T}$ denotes the transpose of $x$, $\mathrm{Tr}$ is the trace operator, and $H_h$ is the Hessian of $h$ with respect to $x$, that is $H_h(x)_{i,j}=\frac{\partial^2}{\partial x_i\partial x_j}h(x)$.

Next, summing over all the particles alive at time $t$, the collection of which we denote by $\mathcal{I}_t$, gives
\begin{equation}\label{motion}
\d\langle h,Z_t\rangle=\left\langle \nabla h(\cdot)\cdot\tilde{b}(\cdot),Z_t\right\rangle\d t+\left\langle \frac{1}{2}\mathrm{Tr}\left[\sigma^{\tt T}(\cdot) H_{h}(\cdot)\sigma(\cdot)\right],Z_t\right\rangle\d t+\sum_{\alpha\in\mathcal{I}_t} (\nabla h(\xi_t^\alpha))^{\tt T} \sigma(\xi_t^\alpha)\d B_t^\alpha,
\end{equation}
where for each $\alpha$, $B^\alpha$ is an independent copy of $B$, and $\xi^\alpha$ is the current position of individual $\alpha\in\mathcal{I}_t$.

If an individual branches at time $t$ then we have
\begin{equation}\label{branching}
\langle h,Z_t-Z_{t-}\rangle=\sum_{\alpha:\textrm{death time of }\alpha =t}(k_\alpha-1)h(\xi_t^\alpha).
\end{equation}
Here $k_\alpha$ is the number of children of individual $\alpha$, which has distribution $\{p_k,k=0,1,\dots\}$.

Simple algebra shows that 
\[
\mathrm{Tr}\left[ \sigma^{\tt T}(x)H_h(x)\sigma(x)\right]=\sum_{ij}a_{ij}(x)\frac{\partial^2}{\partial x_i \partial x_j}h(x),
\]
thus by combining \eqref{motion} and \eqref{branching} we get
\begin{equation}\label{sdeskeleton}
\langle h,Z_t\rangle=\langle h,Z_0\rangle+\int_0^t\langle \mathcal{L}^w h,Z_s\rangle\d s+V_t^c+\int_0^t\int_{E}\int_{\mathbb{N}}\langle h,(k-1)\delta_x\rangle N^\dagger_s( \d s,\d x, \d\{ k\}),
\end{equation}
where $V_t^c$ is a continuous local martingale given by 
\begin{equation}\label{Vc}
V_t^c=\int_0^t \sum_{\alpha\in\mathcal{I}_s} (\nabla h(\xi_s^\alpha))^{\tt T} \sigma(\xi_s^\alpha)\d B_s^\alpha,
\end{equation}
and, $N^\dagger_s$ is an optional random measure on $[0,\infty)\times E\times \mathbb{N}$ with predictable compensator of the form $\hat{N}^\dagger(\d s, \d x, \d \{k\}) = \d s K^\dagger(Z_{s-}, \d x, \d \{k\})$ such that, for $\mu\in\mathcal{M}(E)$,
\begin{equation}
K^\dagger(\mu,\d x, \d \{k\})= \mu(\d x)  q(x)p_k(x) \#( \d\{k\})
\label{nastyrate}
\end{equation}
where $q$, $p_k(x)$ are given by \eqref{rate1}, \eqref{pk} and $\#$ is the counting measure.
 The reader will note that, for a (random) measure $M\in\mathcal{M}_a(E)$, we regularly interchange the notion of  $\sum_{k\in \mathbb{N}} \cdot$ with $\int_{\mathbb{N}}\cdot M(\d\{k\})$.

Note that from \eqref{Vc} it is easy to see that the quadratic variation of $V_t^c$ is 
\[
\langle V^c \rangle_t=\int_0^t \sum_{\alpha\in\mathcal{I}_s} (\nabla h(\xi_s^\alpha))^{\tt T} \sigma(\xi_s^\alpha)\sigma(\xi_s^\alpha)^{\tt T}\nabla h(\xi_s^\alpha)\d s=
 \int_0^t \langle( \nabla h)^{\tt T} a \nabla h,Z_s\rangle\d s.
\]

%
%

\subsection{Thinning of the SDE}
Now we will see how to modify the SDE given by \eqref{zlsde} in order to separate out the different types of immigration processes.
We use ideas developed in \cite{FFK17}.

Recall that the SDE describing the superprocess $(X,\mathbb{P}_\mu)$ takes the following form

\begin{align}\label{sdeloc}
\langle f,X_t\rangle&=\langle f,\mu\rangle+\int_0^t\langle \alpha f,X_s\rangle\d s +M_t^c(f)\notag\\
&\hspace{1cm}+\int_0^t\int_{E}\int_{(0,\infty)}\langle f, u \delta _x \rangle \tilde{N}( \d s,\d x, \d u)+\int_0^t \langle \mathcal{L}f,X_s \rangle\d s,\quad t\geq 0.
\end{align}

Here $M_t^c(f)$ is as in \eqref{zlsde}, and  $N( \d s, \d x, \d u)$ is an optional random measure on $[0,\infty)\times E\times (0,\infty)$ such that,  given $\mu\in\mathcal{M}(E)$, it has predictable compensator given by $\hat{N}( \d s, \d x, \d u)=\d s K(X_{s-}, \d x, \d u)$,where 
\[
K(\mu,  \d x, \d u) = \mu(\d x) m(x,\d u).
\]
Moreover, $\tilde{N}( \d s,\d x, \d u)$ is the associated compensated version of $N( \d s, \d x, \d u)$. 
Denote by $((s_i,x_i, u_i):i\in\mathbb{N})$ some enumeration of the atoms of $N( \d s,\d x, \d u)$.
Next we introduce independent marks to the atoms of $N$, that is we define the random measure
\[
\mathcal{N}( \d s,\d x, \d u,\d\{k\})=\sum_{i\in\mathbb{N}}\delta_{(s_i,x_i, u_i,k_i)}(\d s,\d x, \d u,\d\{k\}),
\]
whose predictable compensator $\d s\mathcal{K}(X_{s-},\d x, \d u,\d\{k\})$ has the property that, for  $\mu\in\mathcal{M}(E)$,
\begin{align*}
&
 \mathcal{K}(\mu,\d x, \d u,\d\{k\})
 = \mu(\d x) m(x,\d u)
  \frac{(w(x)u)^k}{k!}\e^{-w(x)u}\# (\d\{k\}).
\end{align*}
Now we can define three  random measures by
\[
N^0(\d s,\d x, \d u)=\mathcal{N}(\d s,\d x, \d u,\{k=0\}),
\]
\[
N^1(\d s,\d x, \d u)=\mathcal{N}(\d s,\d x, \d u,\{k=1\})
\]
and
\[
N^2(\d s,\d x, \d u)=\mathcal{N}(\d s,\d x, \d u,\{k\geq 2\}).
\]
Using Proposition 10.47 of \cite{J79} we see that $N^0$, $N^1$ and $N^2$ are also optional random measures and their compensators $\d s K^0(X_{s-},\d x, \d u)$,  $\d s K^1(X_{s-},\d x, \d u)$ and  $\d sK^2(X_{s-},\d x, \d u)$ satisfy
\[
 K^0(\mu,\d x, \d u)= \mu(\d x) \e^{-w(x)u}m(x,\d u),
\]
\[
K^1(\mu,\d x, \d u)=\mu(\d x) w(x)u\e^{-w(x)u}m(x,\d u), 
\]
and 
\[
 K^2(\mu,\d x, \d u)=\mu(\d x)\sum_{k=2}^\infty \frac{(w(x)u)^k}{k!}\e^{-w(x)u} m(x,\d u)
\]
for $\mu\in\mathcal{M}(E)$.
Using these processes we can rewrite \eqref{sdeloc}, so we get
\begin{align}
\langle f,X_t\rangle &=\langle f,\mu\rangle +\int_0^t\langle \alpha f,X_s\rangle\d s +M_t^c(f)+\int_0^t\int_E\int_{(0,\infty)}\langle f, u\delta _x \rangle\tilde{N}^0(\d s,\d x, \d u)+\int_0^t \langle \mathcal{L}f,X_s \rangle\d s\notag\\
&\hspace{1.4cm}+\int_0^t\int_E\int_{(0,\infty)}\langle f,u\delta_x \rangle N^1(\d s,\d x, \d u)+\int_0^t\int_E\int_{(0,\infty)}\langle f, u\delta_x \rangle N^2(\d s,\d x, \d u)\notag\\
&\hspace{1.4cm}-\int_0^t\left\langle \int_{(0,\infty)}uf(\cdot)\left(1-\e^{-uw(\cdot)} \right)m(\cdot,\d u),X_{s-}\right\rangle \d s\notag\\
&=\langle f,\mu\rangle -\int_0^t\langle \psi'(\cdot,w(\cdot,s)) f(\cdot),X_s\rangle\d s +M_t^c(f)+\int_0^t\int_E\int_{(0,\infty)}\langle f,u\delta_x\rangle\tilde{N}^0(\d s,\d x, \d u)\notag\\
&\hspace{1.4cm}+\int_0^t \langle \mathcal{L}f,X_s \rangle\d s+\int_0^t\int_E\int_{(0,\infty)}\langle f, u\delta_x \rangle N^1(\d s,\d x, \d u)\notag\\
&\hspace{1.4cm}+\int_0^t\int_E\int_{(0,\infty)}\langle f, u\delta_x \rangle N^2(\d s, \d x, \d u)+\int_0^t\langle 2\beta w f,X_{s-}\rangle\d s,
\label{thinning}
\end{align}
where we have used the fact that $\alpha(x)-\int_{(0,\infty)}(1-\e^{-w(x)u})u m(x,\d u)=-\psi'(x,w(x))+2\beta(x)w(x)$.
Recalling \eqref{psistar1} we see that the first line of the right-hand side of \eqref{thinning} corresponds to the dynamics of a $(\mathcal{P},\psi^*)$-superprocess.
Our aim now is to link the remaining three terms to the three types of immigration along the skeleton, and write down a system of SDEs that describe the skeleton and total mass simultaneously.
Heuristically speaking, this system of SDEs will consist of  \eqref{sdeskeleton} and a second SDE which looks a little bit like \eqref{thinning} (note, the latter has no dependency on the process $Z$ as things stand). To some extent, we can think of the SDE \eqref{thinning} as what one might see when `integrating out' \eqref{sdeskeleton} from  the aforesaid second SDE in the coupled system; indeed this will be one of our main conclusions.


\subsection{Coupled SDE}
Following the ideas of the previous sections we introduce the following driving sources of randomness that we will use in the construction of our coupled SDE. Our coupled system will describe the evolution of  the pair of random measures $(\Lambda, Z) = ((\Lambda_t, Z_t), t\geq 0)$ on $\mathcal{M}(E)\times\mathcal{M}_a(E)$.
\begin{itemize}
\item Let ${\texttt N}^0(\d s,\d x, \d u)$ be an optional random measure on $[0,\infty)\times E\times (0,\infty)$, which depends on $\Lambda$ with predictable compensator  $\hat{{\texttt N}}^0(\d s,\d x, \d u) = \d s\, \texttt{K}^0(\Lambda_{s-}, \d x, \d u)$, where, for $\mu\in \mathcal{M}(E)$, 
\[
\texttt{K}^0(\mu, \d x, \d u)= \mu(\d x) \e^{-w(x)u} m(x,\d u),
\]
and $\tilde{{\texttt N}}^{0}(\d s, \d x, \d u)$ is its compensated version;

\item let ${\texttt N}^1(\d s,\d x, \d u)$ be an optional random measure on $[0,\infty)\times E\times (0,\infty)$, dependent on $Z$, with predictable compensator $\hat{{\texttt N}}^1(\d s,\d x, \d u) = \d s\, \texttt{K}^1(Z_{s-}, \d x, \d u)$ so that, for $\mu\in\mathcal{M}_a(E)$,
\[
\texttt{K}^1(\mu, \d x, \d u)=   \mu(\d x) \e^{-w(x)u} m(x,\d u);
\]
\item  define ${\texttt N}^2(\d s,\d\rho, \d x, \d u)$ an optional random measure on $[0,\infty)\times \mathbb{N}\times E\times (0,\infty)$ also dependent on $Z$, with predictable compensator 
\[\hat{{\texttt N}}^2(\d s,\d\{k\},\d x, \d u)= \d s\, \texttt{K}^2(Z_{s-}, \d\{k\}, \d x, \d u)\] 
so that, for $\mu\in\mathcal{M}_a(E)$,
\[
\texttt{K}^2(\mu, \d\{k\}, \d x, \d u)=  \mu(\d x) q(x) p_k(x)\eta_k(x,\d u) \#(\d\{k\}),
\]
where $q$, $p_k(\d x)$ and $\eta_k(x,\d u)$ are given by \eqref{rate1}, \eqref{pk} and \eqref{etak}.
\end{itemize}

Now we can state our main result.

\begin{theorem}\label{thm}
Consider the following system of SDEs for $f,h\in D_0(\mathcal{L})$,
 {\rm \begin{align}
 \left(
\begin{array}{l}
\langle f,\Lambda_t\rangle\\
 \langle h,Z_t\rangle \\
\end{array}
\right)   =    & \, \left(
\begin{array}{l}
\langle f,\Lambda_0\rangle \\
\langle h, Z_0\rangle \\
\end{array}
\right) - \int_{0}^{t}
 \left( \begin{array}{l}
\langle \partial_z\psi^*(\cdot,0)f,
\Lambda_{s-} \rangle\\
 0 \\
\end{array}
\right) \d s
                      + \left( \begin{array}{l}
U_t^c(f) \\
 V_t^c(h) \\
\end{array}
\right)
  \notag\\
 &  + \int_{0}^{t}\int_{E}\int_{(0,\infty)} \left( \begin{array}{l}
\langle f,u\delta_x\rangle \\
 0 \\
\end{array}
\right)
  \tilde{{\texttt N}}^{0}(\d s, \d x, \d u) 
  + \int_{0}^{t}   \left( \begin{array}{l}
\langle \mathcal{L}f,\Lambda_{s-} \rangle \\
 \langle \mathcal{L}^w h,Z_{s-} \rangle \\
\end{array}
\right) \d s \notag \\
                      &\,   +\int_{0}^{t}\int_{E}\int_{(0,\infty)} \left( \begin{array}{l}
\langle f,u\delta_x\rangle \\
 0 \\
\end{array}
\right){\texttt N}^1(\d s,\d x, \d u)
\notag\\                     &\,  
      + \int_{0}^{t}\int_{\mathbb{N}}\int_{E}\int_{(0,\infty)} \left( \begin{array}{l}
 \langle f,u\delta_x\rangle \\
 \langle h,(k-1)\delta_x \rangle \\
\end{array}
\right) {\texttt N}^{2}(\d s,\d \{k\},\d x, \d u)
\notag\\                     &\,  
   +  \int_{0}^{t}   \left( \begin{array}{l}
\langle 2\beta f,Z_{s-}\rangle \\
 0 \\
\end{array}
\right) \d s, \qquad t\geq 0,              
                          \label{coupledSDE1}                          
    \end{align}}
inducing probabilities $\mathbf{P}_{(\mu,\nu)}$, $\mu\in\mathcal{M}(E)$, $\nu\in\mathcal{M}_a(E)$,    where $(U_t^c(f),t\geq 0)$ is a continuous local martingale with quadratic variation $2 \langle\beta f^2,\Lambda_{t-}\rangle\d t$, and $(V_t^c(h),t\geq 0)$ is a continuous local martingale with quadratic variation $\langle (\nabla h)^{\tt T} a\nabla h,Z_{t-}\rangle\d t$. (Note $\partial_z\psi^*(x,0) = \psi'(x,w(x))$ is another way of identifying the drift term in the first integral above). With an immediate abuse of  notation, write $\mathbf{P}_\mu = \mathbf{P}_{(\mu, {\rm Po}(w\mu))}$, where ${\rm Po}(w\mu)$ is an independent Poisson random measure on $E$ with intensity $w\mu$. 
Then we have the following:
\begin{enumerate} 
\item[(i)] There exists a unique weak solution to the SDE \eqref{coupledSDE1} under each $\mathbf{P}_{(\mu,\nu)}$;
\item[(ii)] Under each $\mathbf{P}_\mu$, for $t\geq 0$, conditional on $\mathcal{F}_t^\Lambda=\sigma(\Lambda_s,s\leq t)$,  $Z_t$  is a Poisson random measure with intensity $w(x)\Lambda_t(\d x)$;
\item[(iii)] the process $(\Lambda_t,t\geq 0)$, with probabilities $\mathbf{P}_{\mu}$, $\mu\in\mathcal{M}(E)$, is Markovian and a weak solution to \eqref{sdeloc}.
\end{enumerate}
\end{theorem}

The rest of the paper is dedicated to the proof of this theorem, which we  split  over several subsections.

\section{Proof of Theorem \ref{thm} (i): existence}\label{(i)existence}\label{4} 
Consider the pair $(\Lambda, Z)$, where $Z$ is a $(\mathcal{P}^w,F)$ branching Markov process with $Z_0=\nu$ for some $\nu\in\mathcal{M}_a(E)$, and whose jumps are coded by the coordinates of the random measure $\mathtt{N}^2$. 
Furthermore we define $\Lambda_t=X^*_t+D_t$, where $X^*$ is an independent copy of the $(\mathcal{P},\psi^*)$-superprocess with initial value $X^*_0=\mu$, $\mu\in\mathcal{M}(E)$, and the process $(D_t,t\geq 0)$ is described by 
\begin{align}
\langle f,D_t\rangle =& \int_0^t\int_{\mathbb{D}([0,\infty)\times \mathcal{M}(E))}\langle f,\omega_{t-s}\rangle \mathtt{N}^1(\d s,\cdot, \cdot,\d\omega)\notag\\
&\hspace{1cm}+\int_0^t\int_{\mathbb{D}([0,\infty)\times \mathcal{M}(E))}\langle f,\omega_{t-s}\rangle \mathtt{N}^2(\d s,\cdot,\cdot, \cdot,\d\omega)\notag\\
&\hspace{2cm}+\int_0^t\int_{\mathbb{D}([0,\infty)\times \mathcal{M}(E))}\langle f,\omega_{t-s}\rangle \mathtt{N}^*(\d s,\d \omega),
\label{D}
\end{align}
where $f\in D_0(\mathcal{L})$ and with a slight abuse of the notation that was introduced preceding Theorem \ref{thm},
\begin{itemize}
\item $\mathtt{N}^1$ is an optional random measure on $[0,\infty)\times E \times(0,\infty)\times\mathbb{D}([0,\infty)\times \mathcal{M}(E))$ with predictable compensator
\[
\hat{\mathtt{N}}^1(\d s,\d x, \d u,\d\omega)=\d s Z_{s-}(\d x)u\e^{-w(x)u}m(x,\d u)\mathbb{P}_{u\delta_x}^*(\d\omega),
\]
\item $\mathtt{N}^2$ is an optional random measure on $[0,\infty)\times \mathbb{N}\times E \times(0,\infty)\times \mathbb{D}([0,\infty)\times \mathcal{M}(E))$ with predictable compensator
\[
\hat{\mathtt{N}}^2(\d s,\d\{k\},\d x,\d u, \d\omega)=\d s Z_{s-}(\d x)q(x)p_k(x)\eta_k(x,\d u)\mathbb{P}_{u\delta_x}^*(\d\omega)\#(\d \{k\}),
\]
\item $\mathtt{N}^*$ is an optional random measure on $[0,\infty)\times \mathbb{D}([0,\infty)\times \mathcal{M}(E))$ with predictable compensator
\begin{equation}
\hat{\mathtt{N}}^*(\d s,\d\omega)=\d s\int_E 2\beta(x) Z_{s-}(\d x)\mathbb{Q}^*_{x}(\d\omega),
\label{Qcomp}
\end{equation}
where 
\[
\mathbb{Q}_x^*(1-\e^{-\langle f,\omega_t\rangle})=-\log\mathbb{E}_{\delta_x}^*(\e^{-\langle f, X_t\rangle})=u_f^*(x,t).
\]
\end{itemize}
Note, we have used $\cdot$ to denote marginalisation so, e.g. 
\[
\mathtt{N}^1(\d s,\cdot, \cdot,\d\omega) = \int_E\int_{(0,\infty)}\mathtt{N}^1(\d s,\d x, \d u,\d\omega)
\]
and, to be consistent with previous notation, we also have e.g.
\[
\mathtt{N}^2(\d s,\d\{k\},\d x,\d u, \cdot) = \mathtt{N}^2(\d s,\d\{k\},\d x,\d u).
\]

We claim that the pair $(\Lambda, Z)$ as defined above solves the coupled system of SDEs \eqref{coupledSDE1}.
To see why, start by noting that $Z$ solves the second coordinate component of  \eqref{coupledSDE1} by definition of it being a spatial branching process; cf. \eqref{sdeskeleton}.  In dealing with the term $\d D_t$, we first note that the random measures $ \mathtt{N}^1$ and $ \mathtt{N}^2$ have finite activity through time, whereas $ \mathtt{N}^*$ has infinite activity. Suppose we write $\mathtt{I}^{(i)}_t$, $i =1,2,3$, for the three integrals on the right-hand side of \eqref{D}, respectively.   Taking the case of $\mathtt{I}^{(1)}_t$,  if $t$ is a jump time of $ \mathtt{N}^1(\d t, \d x, \d u, \d \omega)$, then 
$
\Delta \mathtt{I}^{(1)}_t = \langle f, \omega_0 \rangle ,
$
noting in particular that $ \omega_0 = u\delta_x$. A similar argument produces $\Delta \mathtt{I}^{(2)}_t = \langle f, \omega_0 \rangle = \langle f, u\delta_x\rangle $, when $t$ is a jump time of $ \mathtt{N}^2(\d t ,\d \{k\}, \d x, \d u, \d\omega)$. In contrast, on account of the excursion measures $(\mathbb{Q}^*_{x}, x\in E)$ having the property that $\mathbb{Q}^*_{x}(\omega_0 >0) =0$, we have $\Delta \mathtt{I}^{(3)}_t =0$.  Nonetheless, the structure of the compensator \eqref{Qcomp} implies that  there is a rate of arrival (of these zero contributions) given by
\begin{align*}
\int_{\mathbb{D}([0,\infty)\times \mathcal{M}(E))}\langle f, \omega_0\rangle \hat{\mathtt{N}}^*(\d s, \d\omega)  &=\d s\int_E 2\beta(x) Z_{s-}(\d x)\mathbb{Q}^*_{x}(\langle f, \omega_0\rangle) \\
&= \d s\int_E 2\beta(x) Z_{s-}(\d x)f(x),
\end{align*}
where we have used the fact that $\mathbb{E}^*_{\delta_x}[\langle f, X_t\rangle] = \mathbb{Q}^*_{x}(\langle f, \omega_t \rangle)$, $t\geq 0$; see e.g. \cite{DK04}.

Now suppose that $t$ is not a jump time of $\mathtt{I}^{(i)}_t$, $i =1,2,3$. In that case, we note that $\langle f, \Lambda_{t}\rangle = \langle f,  X^*_{t}  \rangle + \langle f,  D_{t}  \rangle $ is nothing more than the aggregation of mass that has historically immigrated and evolved under $\mathbb{P}^*$. As such (comparing with e.g. \eqref{zlsde})
\begin{align}
\d\langle f, \Lambda_t\rangle =& -\langle\partial_z\psi^*(\cdot,0)f,     \Lambda_{t-}\rangle  \d t + \d U_t^c(f) + \d U_t^d(f)+ \langle\mathcal{L}f,\Lambda_{t-} \rangle\d t \notag\\
&
+\int_E\int_{(0,\infty)}
\langle f, u\delta_x \rangle \mathtt{N}^1(\d t,\d x,\d u)\notag\\
&\hspace{1cm}+\int_{\mathbb{N}}\int_E\int_{(0,\infty)}
\langle f, u\delta_x\rangle \mathtt{N}^2(\d t,\d\{k\}, \d x, \d u)\notag\\
&\hspace{2cm}+
\langle 2\beta f,Z_{t-} \rangle\d t, \qquad t\geq 0,
\label{existsde}
\end{align}
where 
\[
U_t^d(f) = \int_0^t\int_{E}\int_{(0,\infty)}\langle f, u\delta_x\rangle\tilde{N}^0_s( \d s,\d x, \d u), \quad t\geq 0,
\]
and $U^c(f)$ was defined immediately above Theorem  \ref{thm}. As such, we see from \eqref{existsde} that the pair $(\Lambda,Z)$ defined in this section provides a solution to \eqref{coupledSDE1}.

\section{Some integral and differential equations}\label{5}
The key part of our reasoning in proving parts (ii) and (iii) of Theorem \ref{thm} will be to show that
\begin{equation}\label{law}
\mathbf{E}_\mu\left[\e^{-\langle f,\Lambda_t\rangle-\langle h,Z_t\rangle}\right]=\mathbb{E}_\mu\left[ \e^{-\langle f+w (1-\e^{-h}),X_t\rangle}\right],
\end{equation}
where $X$ satisfies \eqref{sdeloc}.
Moreover, the key idea behind  the proof of \eqref{law} is to fix $T>0$ and $f,h\in D_0(\mathcal{L})$, and choose time-dependent test functions $f^T$ and $h^T$ in a way that the processes
\begin{equation}
F_t^T=\e^{-\langle f^T(\cdot,T-t),\Lambda_t\rangle-\langle h^T(\cdot,T-t),Z_t\rangle},\qquad t\in[0,T],
\label{mg1}
\end{equation}
and 
\begin{equation}
G_t^T= \e^{-\langle f^T(\cdot,T-t)+w (1-\e^{-h^T(\cdot,T-t)}),X_t\rangle},\qquad t\in[0,T],
\label{mg2}
\end{equation}
have constant expectations on $[0,T]$.
The test functions  are defined as solutions to some partial differential equations with final value conditions  $f^T(x,T)=f(x)$ and $h^T(x,T)=h(x)$.
This, together with the fact that $\Lambda_0=X_0=\mu$, and that $Z_0$ is a Poisson random measure with intensity $w(x)\Lambda_0(\d x)$, then will give us \eqref{law}.

Thus to prove \eqref{law} and hence parts (ii) and (iii) of Theorem \ref{thm}, we need the existence of solutions of two differential equations. 
Recall from Section \ref{2} that in the skeletal decomposition of superprocesses the total mass present at time $t$ has two main components.
The first one corresponds to an initial burst of subcritical mass, which is an independent copy of $(X,\mathbb{P}^*_\mu)$, and the second one is the accumulated mass from the dressing of the skeleton.
As we will see in the next two results below, one can associate the first differential equation, that is the equation defining $f^T$, to $(X,\mathbb{P}^*_\mu)$, while the equation defining $h^T$ has an intimate relation to the dressed tree defined in the previous section.


\begin{lemma}\label{flemma}
Fix $T>0$, and let $f\in D_0(\mathcal{L})$. Then the following differential equation has a  unique non-negative solution
\begin{align}
\frac{\partial}{\partial t}f^T(x,t)&=-\mathcal{L}f^T(x,t)+\psi^*(x,f^T(x,t)), \quad 0\leq t\leq T\label{fT},\\
f^T(x,T)&=f(x),\nonumber
\end{align}
where $\psi^*$ is given by \eqref{psistar1}.
\end{lemma}
\begin{proof}
Recall that $(X,\mathbb{P}^*_\mu)$ is a $(\mathcal{P}, \psi^*)$-superprocess, and as such its law can be characterised through an integral equation. 
More precisely, for all $\mu\in \mathcal{M}( E)$ and $f\in B^+( E)$, we have
\[
\mathbb{E}^*_\mu\left[\e^{-\langle f, X_t\rangle}\right]=\exp\left\lbrace -\langle u^*_f(\cdot,t)\mu\rangle\right\rbrace,\quad t\geq 0,
\]
where $u^*_f(x,t)$ is the unique non-negative solution to the integral equation
\begin{equation}\label{inteqpsistar}
u^*_f(x,t)=\mathcal{P}_t[f](x)-\int_0^t \d s\cdot \mathcal{P}_s[\psi^*(\cdot,u^*_f(\cdot,t-s))](x),\quad x\in E,\; t\geq 0.
\end{equation}
Li (Theorem 7.11 of \cite{ZL11}) showed that this integral equation is equivalent to the following differential equation
\begin{align}
\frac{\partial}{\partial t}u^*_f(x,t)&=\mathcal{L}u^*_f(x,t)-\psi^*(x,u^*_f(x,t)),\label{diffeqpsistar}\\
u^*_f(x,0)&=f(x).\nonumber
\end{align}
Thus \eqref{diffeqpsistar} also has a unique non-negative solution. 
If for each fixed $T>0$ we define $f^T(x,t)=u_f^*(x,T-t)$, then it is not hard to see that the lemma holds.
\end{proof}

\begin{theorem}\label{htheorem}
Fix $T>0$, and take $f,h\in D_0(\mathcal{L})\cap B^+_b(E)$. 
If $f^T$ is the unique solution to \eqref{fT}, then the following differential equation has a unique non-negative solution
\begin{align}
\e^{-h^T(x,t)}w(x)\frac{\partial}{\partial t}h^T(x,t)=&\mathcal{L}\left( w(x)\e^{-h^T(x,t)}\right)\nonumber\\
&+\left(\psi^*\left(x,-w(x)\e^{-h^T(x,t)}+f^T(x,t)\right)-\psi^*(x,f^T(x,t))\right),\label{hT}\\
h^T(x,T)=&h(x)\nonumber,
\end{align}
where $\psi^*$ is given by \eqref{psistar1}, and $w$ is a martingale function that satisfies the conditions in Section \ref{2}.
\end{theorem}
\begin{proof} Recall the process $( D, Z)$ constructed in Section \ref{(i)existence}.
For every $\mu\in\mathcal{M}(E)$, $\nu\in\mathcal{M}_a(E)$ and $f,h\in B^+_b(E)$ we have
\[
\mathbf{E}_{(\mu,\nu)}\left[ \e^{-\langle f, D_t\rangle-\langle h, Z_t\rangle }\right]=\e^{ -\langle v_{f,h}(\cdot,t),\nu\rangle},
\]
where $\mathrm{exp}\{-v_{f,h}(x,t)\}$ is the unique $[0,1]$-valued solution to the following integral equation
\begin{equation}\label{inteq2}
\begin{split}
&w(x)\e^{-v_{f,h}(x,t)}=\mathcal{P}_t\left[ w(\cdot)\e^{-h(\cdot)}\right](x)\\
&\hspace{1cm}+\int_0^t \mathrm{d}s\cdot \mathcal{P}_{s}\left[\psi^*\left(\cdot,-w(\cdot)\e^{-v_{f,h}(\cdot,t-s)}+u^*_f(\cdot,t-s)\right)-\psi^*(\cdot,u^*_f(\cdot,t-s))\right](x),
\end{split}
\end{equation}
and $u^*_f$ is the unique non-negative solution to \eqref{inteqpsistar}.
Indeed, this claim is  a straightforward adaptation of the proof of Theorem 2 in \cite{KPR14}, the details of which we leave to the reader. Note also that a similar statement has appeared in  \cite{BKM11} in the non-spatial setting.

\medskip

Next suppose that  $f,h\in D_0(\mathcal{L})\cap B^+_b(E)$. 
We want to show that solutions to the integral equation \eqref{inteq2} are equivalent  to solutions of the following differential equation
\begin{align}
\e^{-v_{f,h}(x,t)}w(x)\frac{\partial}{\partial t}v_{f,h}(x,t)=&-\mathcal{L}\left[ w(\cdot)\e^{-v_{f,h}(\cdot,t)}\right](x)\nonumber\\
&-\left(\psi^*\left(x,-w(x)\e^{-v_{f,h}(x,t)}+u^*_f(x,t)\right)-\psi^*(x,u^*_f(x,t))\right),\label{diffeq2}\\
v_{f,h}(x,0)=&h(x).\nonumber
\end{align}
The reader will note that the statement and proof of this claim are classical. However, we include them here  for the sake of completeness. One may find similar computations in e.g. the Appendix of \cite{dynkinappendix}.

We first prove the claim that solutions to the integral equation \eqref{inteq2} are solutions to the differential equation 
\eqref{diffeq2}.
To this end consider \eqref{inteq2}. 
Note that since $\mathcal{P}$ is a Feller semigroup the right hand side is differentiable in $t$, and thus $v_{f,h}(x,t)$ is also differentiable in $t$.
To find the differential version of the equation, we can use the standard technique of propagating the derivative at zero using the semigroup property of $v_{f,h}$ and  $u^*_f$.
Indeed, on one hand the semigroup property can easily be verified using
\begin{align*}
\mathbf{E}_{(\mu,\nu)}\left[ \e^{-\langle f,\Lambda_{t+s}\rangle -\langle h,Z_{t+s}\rangle}\right]&=\mathbf{E}_{(\mu,\nu)}\left[ \mathbf{E}\left[\left. \e^{-\langle f,\Lambda_{t+s}\rangle -\langle h,Z_{t+s}\rangle}\right| \mathcal{F}_t\right]\right]\\
&=\mathbf{E}_{(\mu,\nu)}\left[ \mathbf{E}_{(\Lambda_t,Z_t)}\left[ \e^{-\langle f,\Lambda_s\rangle -\langle h,Z_s\rangle}\right]\right]\\
&=\mathbf{E}_{(\mu,\nu)}\left[\e^{-\langle u_f^*(\cdot,t),\Lambda_s \rangle-\langle v_{f,h}(\cdot,t),Z_s\rangle} \right]\\
&=\e^{-\left\langle u^*_{u^*_f(\cdot,t)}(\cdot,s),\mu \right\rangle-\left\langle v_{u^*_f(\cdot,t),v_{f,h}(\cdot,t)}(\cdot,s),\nu\right\rangle},
\end{align*}
 that is we have $v_{u^*_f(\cdot,t),v_{f,h}(\cdot,t)}(\cdot,s)=v_{f,h}(\cdot,t+s)$, and  $u^*_{u_f^*(\cdot,t)}(\cdot,s)=u_f^*(\cdot,t+s)$.
 This implies 
 \begin{equation}
\frac{\partial}{\partial t}u_f^*(x,t+)=\left.\frac{\partial}{\partial s}u_{u_f^*(\cdot,t)}(x,s)\right|_{s\downarrow 0}=\frac{\partial}{\partial s}u_{u_f^*(\cdot,t)}(x,0+),
\label{obs1}
\end{equation}
and
\begin{equation}
\frac{\partial}{\partial t}v_{f,h}(\cdot,t+)=\left.\frac{\partial}{\partial s}v_{u^*_f(\cdot,t),v_{f,h}(\cdot,t)}(x,s)\right|_{s\downarrow 0},
\label{obs2}
\end{equation}
providing that the two derivatives at zero exist from the right. One may similarly use the semigroup property, splitting at time $s$ and $t-s$ to give the left derivatives at time $t>0$.
On the other hand differentiating \eqref{inteq2} in $t$ and taking $t\downarrow 0$ gives
\begin{align}\label{diffzero}
-w(x)\e^{-v_{f,h}(x,0+)}\left.\frac{\partial}{\partial t}v_{f,h}(x,t)\right|_{t=0+}=&\mathcal{L}\left[ w(\cdot)\e^{-h(\cdot)}\right](x)\nonumber\\
&+\psi^*\left( x,-w(x)\e^{-v_{f,h}(x,0+)}+u_f^*(x,0+) \right)\\
&-\psi^*(x,u_f^*(x,0+))\nonumber,
\end{align}
which, recalling $v_{f,h}(x,0)=h(x)$ and $u_f^*(x,0)=f(x)$, can be rewritten as 
\begin{align*}
\frac{\partial}{\partial t}v_{f,h}(x,0+)=&-\frac{1}{w(x)}\e^{h(x)}\mathcal{L}\left[ w(\cdot)\e^{-h(\cdot)}\right](x)\\
&-\frac{1}{w(x)}\e^{h(x)}\psi^*\left( x,-w(x)\e^{-h(x)}+f(x)\right)+\frac{1}{w(x)}\e^{h(x)}\psi^*(x,f(x)).
\end{align*}
Hence combining the previous observations in \eqref{obs1} and \eqref{obs2}, we get
\begin{align*}
\frac{\partial}{\partial t}v_{f,h}(x,t)=&-\frac{1}{w(x)}\e^{v_{f,h}(x,t)}\mathcal{L}\left[ w(\cdot)\e^{-v_{f,h}(x,t)}\right](x)\\
&-\frac{1}{w(x)}\e^{v_{f,h}(x,t)}\psi^*\left( x,-w(x)\e^{-v_{f,h}(x,t)}+u_f^*(x,t)\right)\\
&+\frac{1}{w(x)}\e^{v_{f,h}(x,t)}\psi^*(x,u_f^*(x,t)).
\end{align*}

\medskip 

To see why the differential equation \eqref{diffeq2} implies the integral equation \eqref{inteq2} define
\[
g(x,s)=\mathcal{P}_{t-s}\left(w(x)\e^{-v_{f,h}(x,s)}\right), \quad 0\leq s\leq t.
\]
Then differentiating with respect to the time parameter gives
\begin{align*}
\frac{\partial}{\partial s}g(x,s)&=-\mathcal{P}_{t-s}w(x)\e^{v_{f,h}(x,s)}\frac{\partial}{\partial s}v_{f,h}(x,s)-\mathcal{P}_{t-s}\mathcal{L}\left(w(x)\e^{-v_{f,h}(x,s)}\right)\\
&=\mathcal{P}_{t-s}\left[\psi^*\left(x,-w(x)\e^{-v_{f,h}(x,s)}+u^*_f(x,s)\right)-\psi^*(x,u^*_f(x,s))\right],
\end{align*}
which then we can integrate over $[0,t]$ to get \eqref{inteq2}.

\medskip

To complete the proof, we fix $T>0$, and define $h^T(x,t):= v_{f,h}(x,T-t)$, and the result follows.
\end{proof}

\section{Proof of Theorem \ref{thm} (ii) and (iii)}\label{6}
The techniques we use here are similar in spirit to those in the proof of Theorem 2.1 in \cite{FFK17}, in a sense that we use stochastic calculus to show the equality \eqref{law}; however what is new in the current setting is the use of the processes \eqref{mg1} and \eqref{mg2}.

\medskip

Fix $T>0$, and let $f^T$ be the unique non-negative solution to \eqref{fT}, and $h^T$ be the unique non-negative solution to \eqref{hT}.
Define $F_t^T:= \e ^{-\langle f^T(\cdot,t),\Lambda_t\rangle-\langle h^T(\cdot,t),Z_t\rangle}$, $t\leq T$.
Using stochastic calculus, we first verify that our choice of $f^T$ and $h^T$ results in the process $F_t^T$, $t\leq T$, having constant expectation on $[0,T]$.
In the definition of $F^T$ both $\langle f^T(\cdot,t),\Lambda_t\rangle$ and $\langle h^T(\cdot,t),Z_t\rangle$ are semi-martingales, thus we can use It\^{o}'s formula (see e.g. Theorem 32 in \cite{P90}) to get 
\begin{equation*}
\begin{split}
\d F_t^T=&- F_{t-}^T\d\Lambda_t^{f^T}- F_{t-}^T\d Z_t^{h^T}+\frac{1}{2} F_{t-}^T\d\left[\Lambda^{f^T},\Lambda^{f^T}\right]_t^c+\frac{1}{2} F_{t-}^T\d\left[Z^{h^T},Z^{h^T}\right]_t^c\\
&+ F_{t-}^T\d\left[\Lambda^{f^T},Z^{h^T}\right]_t^c+\Delta F_t^T+ F_{t-}^T\Delta\Lambda_t^{f^T}+ F_{t-}^T\Delta Z_t^{h^T},\quad 0\leq t\leq T,
\end{split}
\end{equation*}
where $\Delta \Lambda_t^{f^T}=\langle f^T(\cdot,t),\Lambda_t-\Lambda_{t-}\rangle$, and to avoid heavy notation we have written $\Lambda_t^{f^T}$ instead of $\langle f^T(\cdot,t),\Lambda_t\rangle$, and $Z_t^{h^T}$ instead of $\langle h^T(\cdot,t),Z_t\rangle$.
Note that without the movement $Z$ is a pure jump process, and since the interaction between $\Lambda$ and $Z$ is limited to the time of the immigration events, we have that $\left[\Lambda^{f^T},Z^{h^T}\right]_t^c=0$.
Taking advantage of
\[
F_t^T=F_{t-}^T\e^{-\Delta\Lambda_t^{f^T}-\Delta Z_t^{h^T}},
\]
we may thus write in integral form
\begin{equation*}
\begin{split}
F_t^T&=F_0^T-\int_0^t F_{s-}^T\d\Lambda^{f^T}_s-\int_0^t F_{s-}^T\d Z^{h^T}_s+\int_0^t F_{s-}^T\langle \beta(\cdot) (f^T(\cdot,s))^2,\Lambda_{s-}\rangle \d s\\
&+\int_0^t F_{s-}^T\langle(\nabla h^T(\cdot,s))^{\tt T} a\nabla h^T(\cdot,s),Z_{s-}\rangle \d s
+\sum_{s\leq t}\left\lbrace \Delta F_s^T+F_{s-}^T\Delta\Lambda_s^{f^T}+ F_{s-}^T\Delta Z_s^{h^T}\right\rbrace.
\end{split}
\end{equation*}
To simplify the notation we used that both $f^T(x,t)$ and $h^T(x,t)$ are continuous in $t$, thus $f^T(x,t)=f^T(x,t-)$ and $h^T(x,t)=h^T(x,t-)$.

We can split up the last term, that is the sum of discontinuities according to the optional random measure in \eqref{coupledSDE1} responsible for this discontinuity.
Thus, writing $\Delta^{(i)},\;i=0,1,2$, to mean an increment coming from each of the three random measures,
\begin{equation*}
\begin{split}
F_t^T=&F_0^T-\int_0^t F_{s-}^T\d\Lambda^{f^T}_s-\int_0^t F_{s-}^T\d Z^{h^T}_s+\int_0^t F_{s-}^T\langle \beta(\cdot) (f^T(\cdot,s))^2,\Lambda_{s-}\rangle \d s\\
&+\int_0^t F_{s-}^T\langle(\nabla h^T(\cdot,s))^{\tt T} a\nabla h^T(\cdot,s),Z_{s-}\rangle \d s+\sum_{s\leq t}F_{s-}^T\left\lbrace \e^{-\Delta^{(0)}\Lambda_s^{f^T}}-1+\Delta^{(0)}\Lambda_s^{f^T}\right\rbrace\\
&+\sum_{s\leq t}F_{s-}^T\left\lbrace \e^{-\Delta^{(1)}\Lambda_s^{f^T}}-1+\Delta^{(1)}\Lambda_s^{f^T}\right\rbrace\\
&+\sum_{s\leq t}F_{s-}^T\left\lbrace  \e^{-\Delta^{(2)}\Lambda_s^{f^T}-\Delta Z_s^{h^T}}-1+\Delta^{(2)}\Lambda_s^{f^T}+ \Delta Z_s^{h^T}\right\rbrace.
\end{split}
\end{equation*}
Next, plugging in $\d\Lambda_s^{f^T}$ and $\d Z_s^{h^T}$ gives
\begin{equation}\label{FtT}
\begin{split}
F_t^T=&F_0^T+\int_0^t F_{s-}^T\langle \psi'(\cdot,w(\cdot))f^T(\cdot,s),\Lambda_{s-}\rangle \d s+\int_0^t F_{s-}^T\langle \beta(\cdot)(f^T(\cdot,s))^2,\Lambda_{s-}\rangle\d s\\
&-\eta\int_0^t F_{s-}^T\langle \mathcal{L}f^T(\cdot,s),\Lambda_{s-}\rangle\d s
-\int_0^t F_{s-}^T\left\langle \frac{\partial}{\partial s}f^T(\cdot,s),\Lambda_{s-}\right\rangle\d s\\
&-\int_0^t F_{s-}^T\left\langle \frac{\partial}{\partial s}h^T(\cdot,s),Z_{s-}\right\rangle\d s-\int_0^t F_{s-}^T\langle \mathcal{L}^w h^T(\cdot,s),Z_{s-}\rangle\d s\\
 &-\int_0^t F_{s-}^T \langle 2\beta(\cdot) f^T(\cdot,s),Z_{s-}\rangle \d s+\sum_{s\leq t}F_{s-}^T\left\lbrace \e^{-\Delta^{(0)}\Lambda_s^f}-1+\Delta^{(0)}\Lambda_s^{f^T}\right\rbrace\\
 &+\sum_{s\leq t}F_{s-}^T\left\lbrace \e^{-\Delta^{(1)}\Lambda_s^{f^T}}-1\right\rbrace
 +\sum_{s\leq t}F_{s-}^T\left\lbrace  \e^{-\Delta^{(2)}\Lambda_s^{f^T}-\Delta Z_s^{h^T}}-1\right\rbrace\\
 &+\int_0^t F_{s-}^T\langle (\nabla h^T(\cdot,s))^{\tt T} a\nabla h^T(\cdot,s),Z_{s-}\rangle \d s+M_t^{loc},
\end{split}
\end{equation}
where $M_t^{loc}$ is a local martingale corresponding to the terms $U_t^c(f^T)$, $V_t^c(h^T)$ and the integral with respect to the random measure $\tilde{{\texttt N}}^{0}$ in \eqref{coupledSDE1}.
Note that the two terms with the time-derivative are due to the extra time dependence of the test-functions in the integrals $\langle f^T(\cdot,s),\Lambda_s\rangle$ and $\langle h^T(\cdot,s),Z_s\rangle$. 
In particular a change in $\langle f^T(s,\cdot),\Lambda_s\rangle$ corresponds to either a change in $\Lambda_s$ or a change in $f^T(\cdot,s)$.

Next we show that the local martingale term is in fact a real martingale, which will then disappear when we take expectations.
First note that due to the boundedness of the drift and diffusion coefficients of the branching mechanism, and the conditions we had on its L\'evy measure, the branching of the superprocess can be stochastically dominated by a finite mean CSBP.
This means that the CSBP associated to the Esscher-transformed branching mechanism $\psi^*$ is almost surely finite on any finite time interval $[0,T]$, and thus the function $f^T$ is bounded on $[0,T]$.
Using the boundedness of $f^T$ and the drift coefficient $\beta$, the quadratic variation of the integral 
\begin{equation}\label{intU}
\int_0^tF_{s-}^T \d U_s^c(f^T)
\end{equation}
can be bounded from above as follows
\begin{align*}
\int_0^t 2 F_{s-}^T \langle \beta(\cdot) (f^T(\cdot,s))^2,\Lambda_{s-}\rangle \d s &\leq \int_0^t \e^{-\langle f^T(\cdot,s),\Lambda_{s-}\rangle} \langle C (f^T(\cdot,s))^2,\Lambda_{s-}\rangle \d s\\
&\leq \int_0^t \e^{-\widetilde{C}||\Lambda_{s-}||}\widehat{C}||\Lambda_{s-}||\d s,
\end{align*}
where $C,\widehat{C}$ and $\widetilde{C}$ are finite constants.
Since the function $x\mapsto \e^{-\widetilde{C}x}x$ is bounded on $[0,\infty)$, the previous quadratic variation is finite, and so the process \eqref{intU} is a martingale on $[0,T]$.

To show the martingale nature of the stochastic integral 
\begin{equation}\label{intV}
\int_0^t F_{s-}\d V_s^c(h^T)
\end{equation}
we note that due to construction, $h^T\in \mathcal{D}_0(\mathcal{L})$, and is bounded on $[0,T]$. 
Thus, $V_t^c(h^T)$ is in fact a martingale on $[0,T]$, and since $F_{s-}\leq 1$, $s\in[0,T]$, the quadratic variation of \eqref{intV} is also finite, which gives the martingale nature of \eqref{intV} on $[0,T]$.

Finally, we consider the integral
\begin{equation}\label{intN0}
\int_0^t\int_{E}\int_{(0,\infty)} F^T_{s-} \langle f^T(\cdot,s),u\delta_x\rangle \tilde{N}^0(\d s, \d x, \d u). 
\end{equation}
Note that  for compactly supported $\mu\in\mathcal{M}(E)^\circ$
\begin{align*}
Q_t&:= \mathbb{E}_\mu\left[\int_0^t\int_{E}\int_{(0,\infty)} \left(F^T_{s-}\langle f^T(\cdot,s),u\delta \rangle\right)^2\hat{N}^0(\d s,\d x, \d u)\right]\\
&= \mathbb{E}_\mu\left[\int_0^t \int_E\int_{(0,\infty)}\left(F^T_{s-} u f^T(x,s)\right)^2 \e^{-w(x)u}m(x,\d u) \Lambda_{s-}(\d x)\d s\right]\\
&\leq  \mathbb{E}_\mu\left[\int_0^t \e^{-2C ||\Lambda_{s-}||}C\left\langle \int_{(0,\infty)}u^2\e^{-w(x)u}m(x,\d u),\Lambda_{s-}\right\rangle \d s\right]\\
&\leq  \mathbb{E}_\mu\left[\int_0^t \e^{-\widetilde{C}||\Lambda_{s-}||} \widehat{C}||\Lambda_{s-}||\d s\right]\\
&\leq C' t
\end{align*}
where $C, \widetilde{C},\widehat{C}$ and $C'$ are finite constants.
Thus $Q_t<\infty$ on $[0,T]$, and we can refer to page 63 of \cite{IW89} to conclude that the process \eqref{intN0} is indeed a martingale on $[0,T]$.

Thus, after taking expectations and gathering terms in \eqref{FtT}, we get
\begin{align}\label{EFT}
\mathbf{E}_\mu\left[F_t^T\right]=\mathbf{E}_\mu\left[F_0^T\right]&+\int_0^t \mathbf{E}_\mu\left[F_{s-}^T\langle A(\cdot,f^T(\cdot,s)),\Lambda_{s-}\rangle\right]\d s\nonumber\\
&-\int_0^t \mathbf{E}_\mu\left[F_{s-}^T\left\langle \frac{\partial}{\partial s}f^T(\cdot,s),\Lambda_{s-}\right\rangle\right]\d s \nonumber\\
&+\int_0^t \mathbf{E}_\mu[F_{s-}^T\langle B(\cdot,h^T(\cdot,s),f^T(\cdot,s)),Z_{s-}\rangle]\d s\\
&-\int_0^t \mathbf{E}_\mu\left[F_{s-}^t\left\langle \frac{\partial}{\partial s}h^T(\cdot,s),Z_{s-}\right\rangle\right]\d s,\quad 0\leq t\leq T, \nonumber
\end{align}
where 
\begin{align}\label{A}
A(x,f)&=\psi'(x,w(x))f+\beta(x)f^2-\mathcal{L}f+\int_{(0,\infty)}\left(\e^{-u f}-1+u f\right)\e^{-w(x)u}m(x,\d u)\\
&=-\mathcal{L}f+\psi^*(x,f)\nonumber,
\end{align}
and 
\begin{align}\label{B}
B(x,h,f)&=(\nabla h)^{\tt T}a\nabla h-\mathcal{L}^w h-2\beta(x) f+\int_{(0,\infty)}(\e^{- u f}-1) u\e^{-w(x) u}m(x,\d u)\nonumber\\
&+\sum_{k=2}^\infty \int_{(0,\infty)}\left(\e^{- u f-(k-1)h}-1\right)\frac{1}{w(x)}
\Bigg\{ \beta(x) w^2(x)\delta_0(\d u)\mathbf{1}_{\{ k=2\}}\\
&\hspace{7cm}+w^k(x)\frac{u^k}{k!}\e^{-w(x)u}m(x,\d u)\Bigg\}.\nonumber
\end{align}
We can see immediately that $A(x,f^T(x,t))$ is exactly what we have on the right-hand side of \eqref{fT}.
Furthermore, using that
\begin{equation}\label{id}
(\nabla h)^{\tt T}a\nabla h-\mathcal{L}^w h=\e^h\frac{1}{w}\mathcal{L}\left( w \e^{-h}\right)-\frac{1}{w}\psi(\cdot,w),
\end{equation}
we can also verify that 
\[
B(x,h,f)=\e^{h}\frac{1}{w}\mathcal{L}\left( w\e^{-h}\right)+\e^{h}\frac{1}{w}\left(\psi^*\left( x,-w(x)\e^{-h}+f\right) -\psi^*(x,f)\right),
\]
that is, $B(x,h^T(x,t),f^T(x,t))$ equals to the right-hand side of \eqref{hT}.
Hence, recalling the defining equations of $f^T$ \eqref{fT} and $h^T$ \eqref{hT}, we get that the last four terms of \eqref{EFT} cancel, and thus $\mathbf{E}_\mu[F_t^T]=\mathbf{E}_\mu[F_0^T]$ for $t\in[0,T]$, as required.
In particular, using the boundary conditions for $f^T$ and $h^T$, we get that
\begin{align}\label{lawF}
\mathbf{E}_\mu\left[F_T^T\right]=\mathbf{E}_\mu\left[ \e^{-\langle f,\Lambda_T\rangle-\langle h,Z_T\rangle}\right]
=\mathbf{E}_\mu\left[ \e^{-\langle f^T(\cdot,0),\Lambda_0\rangle -\langle h^T(\cdot,0),Z_0\rangle}\right]=
\mathbf{E}_\mu\left[F_0^T\right].
\end{align}
Note that by construction we can relate the right-hand side of this previous expression to the superprocess.
In particular, using the Poissonian nature of $Z_0$, and that $X_0=\Lambda_0=\mu$ is deterministic we have 
\begin{equation}\label{law0}
\mathbf{E}_\mu\left[\e^{-\langle f^T(\cdot,0),\Lambda_0\rangle-\langle h^T(\cdot,0),Z_0\rangle}\right]=\mathbb{E}_\mu\left[ \e^{-\left\langle f^T(\cdot,0)+w(\cdot) \left(1-\e^{-h^T(\cdot,0)}\right),X_0\right\rangle}\right],
\end{equation}
where $X_t$ is a solution to \eqref{thinning}.
Thus, by choosing the right test-functions, we could equate the value of $F_t^T$ at $T$ to its initial value, which in turn gave a connection with the superprocess.
The next step is to show that the process
\[
\e^{-\left\langle f^T(\cdot,t)+w(\cdot)\left( 1-\e^{-h^T(\cdot,t)}\right),X_t\right\rangle}, \quad t\in[0,T],
\]
has constant expectation on $[0,T]$,
which would then allow us to deduce
\[
\mathbf{E}_\mu\left[ \e^{-\langle f,\Lambda_T\rangle-\langle h ,Z_T\rangle}\right]=\mathbb{E}_\mu\left[ \e^{-\left\langle f +w  \left(1-\e^{-h }\right),X_T\right\rangle}\right].
\]

To simplify the notation let $\kappa^T(x,t):=f^T(x,t)+w(x)\left( 1-\e^{-h^T(x,t)}\right)$, and define $G_t^T:=\e^{-\langle \kappa^T(\cdot,t),X_t\rangle}$.
%
As the argument here is the exact copy of the previous analysis, we only give the main steps of the calculus, and leave it to the reader to fill in the gaps.

Since $\langle \kappa^T(\cdot,t),X_t\rangle$, $t\leq T$, is a semi-martingale, we can use It\^{o}'s formula to get
\begin{equation}\label{GtT}
\begin{split}
G_t^T=G_0^T&+\int_0^t G_{s-}^T\left\langle \psi'(\cdot,w(\cdot))\kappa^T(s,\cdot),X_{s-}\right\rangle \d s+\int_0^t G_{s-}^T\langle \beta(\cdot)(\kappa^T(\cdot,s))^2,X_{s-}\rangle\d s\\
&-\int_0^t  G_{s-}^T\langle 2\beta(\cdot) w(\cdot) \kappa^T(\cdot,s),X_{s-}\rangle \d s
-\int_0^t G_{s-}^T\langle \mathcal{L}\kappa^T(\cdot,s),X_{s-}\rangle \d s\\
&+\int_0^t G_{s-}^T\left\langle \int_0^\infty \left( \e^{-u\kappa^T(\cdot,s)}-1+u\kappa^T(\cdot,s)\right)\e^{-w(\cdot)u}m(\cdot,\d u),X_{s-}\right\rangle\d s\\
&+\int_0^t G_{s-}^t\left\langle \int_0^\infty \left( \e^{-u\kappa^T(\cdot,s)}-1\right)w(\cdot)u\e^{-w(\cdot)u}m(\cdot,\d u),X_{s-}\right\rangle\d s\\
&+\int_0^t G_{s-}^T\left\langle \int_0^\infty \left( \e^{-u \kappa^T(\cdot,s)}-1\right)\sum_{k=2}^\infty \frac{(w(\cdot)u)^k}{k!}\e^{-w(\cdot) u}m(\cdot,\d u),X_{s-}\right\rangle\d s\\
&-\int_0^t G_{s-}^T\left\langle \frac{\partial}{\partial s}\kappa^T(\cdot,s),X_{s-}\right\rangle \d s+M_t^{loc}.
\end{split}
\end{equation}
where $M_t^{loc}$ is a local martingale corresponding to the term $M_t^c(f)$, and the integral with respect to the random measure $\tilde{N}^{0}$ in \eqref{thinning}.
Note that the reasoning that led to the martingale nature of the local martingale term of \eqref{FtT} can also be applied here, which gives that $M_t^{loc}$ in \eqref{GtT} is in fact a true martingale on $[0,T]$, which we denote by $M_t$.

Next we plug in $\kappa^T$, and after some laborious amount of algebra get
\begingroup
\allowdisplaybreaks
\begin{align*}
G_t^T=&G_0^T+\int_0^t G_{s-}^T\langle \psi'(\cdot,w(\cdot))f^T(\cdot,s),X_{s-}\rangle\d s+\int_0^t G_{s-}^T\langle \beta(\cdot)(f^T(\cdot,s))^2,X_{s-}\rangle\d s\\
&-\int_0^t G_{s-}^T\langle \mathcal{L}f^T(\cdot,s),X_{s-}\rangle\d s\\
&+\int_0^t G_{s-}^T\left\langle \int_{(0,\infty)}(\e^{-u f^T(\cdot,s)}-1+u f^T(\cdot,s))\e^{-w(\cdot)u}m(\cdot,\d u),X_{s-} \right\rangle\d s\\
&-\int_0^t G_{s-}^T\langle 2\beta(\cdot) f^T(\cdot,s)\e^{-h^T(\cdot,s)}w(\cdot),X_{s-} \rangle\d s\\
&+\int_0^t G_{s-}^T\left\langle\int_{(0,\infty)}(\e^{- u f^T(\cdot,s)}-1) u\e^{-w(\cdot) u}m(\cdot,\d u)\e^{-h^T(\cdot,s)}w(\cdot),X_{s-}\right\rangle \d s\\
&+\int_0^t G_{s-}^T\left\langle \sum_{k=2}^\infty \int_{(0,\infty)}(\e^{- u f^T(\cdot,s)-(k-1)h^T(\cdot,s)}-1)\frac{1}{w(\cdot)}\right.\\
&\hspace{1cm} \left.\left\lbrace \beta(\cdot) w^2(\cdot)\delta_0(\d u)\mathbf{1}_{\{ k=2\}}+w^k(\cdot)\frac{u^k}{k!}\e^{-w(\cdot)u}m(\cdot,\d u)\right\rbrace \e^{-h^T(\cdot)}w(\cdot), X_{s-}\right\rangle\d s\\
&+\int_0^t G_{s-}^T\left\langle (1-\e^{-h^T(\cdot,s)})\psi(\cdot,w(\cdot))-\mathcal{L}w(\cdot)(1-\e^{-h^T(\cdot,s)}), X_{s-}\right\rangle \d s\\
&-\int_0^t G_{s-}^T\left\langle \frac{\partial}{\partial s}\kappa^T(\cdot,s),X_{s-}\right\rangle \d s+M_t.
\end{align*}
\endgroup
Using once again the identity \eqref{id}, and taking expectations give
\begin{align}\label{EGT}
\mathbb{E}_{\mu}[G_t^T]=\mathbb{E}_{\mu}[G_0^T]&+\int_0^t \mathbb{E}_{\mu}[G_{s-}^T\langle A(\cdot,f^T(\cdot,s)),X_{s-}\rangle]\d s\\
&+\int_0^t \mathbb{E}_{\mu}[G_{s-}^T\langle \e^{-h^T(\cdot,s)}w(\cdot) B(\cdot,h^T(\cdot,s),f^T(\cdot,s), X_{s-}\rangle]\d s \nonumber\\
&-\int_0^t \mathbb{E}_{\mu}\left[G_{s-}^T\left\langle \frac{\partial}{\partial s}\kappa^t(s,\cdot),X_{s-}\right\rangle\right] \d s,\nonumber
\end{align}
where $A$ and $B$ are given by \eqref{A} and \eqref{B}.
Finally, noting
\begin{equation*}
\frac{\partial}{\partial s}\kappa^T(x,s)=\frac{\partial}{\partial s}f^T(x,s)+w(x)\e^{-h^T(x,s)}\frac{\partial}{\partial s}h^T(x,s),
\end{equation*}
gives
\[
\frac{\partial}{\partial s}\kappa^T(s,x)=-A(x,f^T(x,s))-w(x)\e^{-h^T(x,s)}B(x,h^T(x,s),f^T(x,s)),
\]
which results in the cancellation of the last three terms in \eqref{EGT}, and hence verifies the constant expectation of $G_t^T$ on $[0,T]$.
In particular, we have proved that
\begin{equation}\label{lawG}
\begin{split}
\mathbb{E}_\mu[G_T^T]&=\mathbb{E}_\mu\left[ \e^{-\left\langle f+w\left(1-\e^{-h}\right),X_T\right\rangle}\right]\\
&=\mathbb{E}_\mu\left[ \e^{-\left\langle f^T(\cdot,0)+w\left(1-\e^{-h^T(\cdot,0)}\right),X_0\right\rangle}\right]=\mathbb{E}_\mu[G_0^T].
\end{split}
\end{equation}
In conclusion, combining the previous observations \eqref{lawF} and \eqref{law0} with \eqref{lawG} gives
\[
\mathbf{E}_\mu\left[ \e^{-\langle f,\Lambda_T\rangle-\langle h,Z_T\rangle}\right]=\mathbb{E}_\mu\left[ \e^{-\left\langle f+w \left(1-\e^{-h}\right),X_T\right\rangle}\right].
\]
Since $T>0$ was arbitrary the above equality holds for any time $T>0$.

Then we have the following implications.
First, by setting $h=0$ we find that
\[
\mathbf{E}_\mu\left[ \e^{-\langle f,\Lambda_T\rangle}\right]=\mathbb{E}_\mu\left[ \e^{-\left\langle f,X_T\right\rangle}\right],
\]
which not only shows that under our conditions $(\Lambda_t,t\geq 0)$ is Markovian, but also that its semigroup is equal to the semigroup of $(X,\mathbb{P}_\mu)$, and hence proves that $(\Lambda_t,t\geq 0)$ is indeed a weak solution to \eqref{sdeloc}.

Furthermore, choosing $h$ and $f$ not identical to zero, we get that the pair $(\Lambda_t,Z_t)$ under $\mathbf{P}_{\mu}$ has the same law as $(X_t,\mathrm{Po}(w(x)X_t(\d x)))$ under $\mathbb{P}_\mu$, where $\mathrm{Po}(w(x)X_t(\d x))$ is an autonomously independent Poisson random measure with intensity $w(x)X_t(\d x)$, thus $Z_t$ given $\Lambda_t$ is indeed a Poisson random measure with intensity $w(x)\Lambda_t(\d x)$.\hfill$\square$

\section{Proof of Theorem \ref{thm} (i): uniqueness} If we review the calculations that lead to \eqref{lawF}, we observe that any solution $(\Lambda, Z)$ to the coupled SDE \eqref{coupledSDE1} has the property that, for $\mu\in\mathcal{M}(E)$ and $\nu\in\mathcal{M}_a(E)$, 
\[
\mathbf{E}_{(\mu, \nu)} \left[F_T^T\right]=\mathbf{E}_{(\mu, \nu)}\left[ \e^{-\langle f,\Lambda_T\rangle-\langle h,Z_T\rangle}\right]
=\e^{-\langle f^T(\cdot,0),\mu\rangle -\langle h^T(\cdot,0),\nu\rangle}.
\]
Hence, since any two solutions to \eqref{coupledSDE1} are Markovian, the second equality above identifies their transitions as equal thanks to the uniqueness of the PDEs in Lemma \ref{flemma} and Theorem \ref{htheorem}. In other words, there is a unique weak solution to \eqref{coupledSDE1}.











\acks

AEK  acknowledges support from EPSRC grant  EP/L002442/1. DF is supported by a scholarship from the EPSRC Centre for Doctoral Training, SAMBa. JF acknowledges support from Basal-Conicyt Centre for Mathematical Modelling AFB 170001 and Millenium Nucleus  SMCDS. Part of this work was carried out whilst AEK was visiting the Centre for Mathematical Modelling, Universidad de Chile and JF was visiting the Department of Mathematical Sciences at the University of Bath, each is grateful to the host institution of the other for their support.  The authors are additionally grateful to two anonymous referees and an AE for their careful reading and assessment of an earlier version of this paper, which led to significant improvements. 

%
%
%
%


\end{document}